\documentclass[10pt]{amsart}
\usepackage{amscd} 
\usepackage{amsfonts} 
\usepackage{amssymb} 
\usepackage{latexsym} 
\usepackage[arrow, matrix, curve]{xy}

\newcommand{\ncm}{\newcommand}

\newtheorem{theorem}{Theorem}[section]
\newtheorem{prop}[theorem]{Proposition}
\newtheorem{lemma}[theorem]{Lemma}
\newtheorem{cor}[theorem]{Corollary}
\newtheorem{lem&def}[theorem]{Lemma \& Definition}
\newtheorem{definition}[theorem]{Definition}
\newtheorem{example}[theorem]{Example}

\def\C{\mathbb{C}\,} 
\def\Z{\mathbb{Z}\,}

\ncm{\End}{\mbox{\rm End}\,}

\def\Hom{\mbox{\rm Hom}\,}

\def\|{\, | \,}
\def\h{\, \sim \,}

\def\into{\hookrightarrow}
\def\to{\rightarrow}
\ncm{\Res}{\mbox{\rm Res}}

\def\Ind{\mbox{\rm Ind}\,}
\def\Coind{\mbox{\rm CoInd}\,}
\def\Res{\mbox{\rm Res}\,}
\def\o{\otimes}   
    
\def\bra{\langle}
\def\ket{\rangle}

\ncm{\rarr}[1]{\stackrel{#1}{\longrightarrow}}
\ncm{\larr}[1]{\stackrel{#1}{\longleftarrow}}
\def\cop{\Delta}

\def\eps{\varepsilon}

\def\-2{_{(-2)}}
\def\-1{_{(-1)}}
\def\0{_{(0)}}
\def\1{_{(1)}}
\def\2{_{(2)}}
\def\3{_{(3)}}

\def\du1{\hat 1}

\begin{document}
\title[H-depth]{Odd H-depth and H-separable extensions}
\author{Lars Kadison} 
\address{Departamento de Matematica \\ Faculdade de Ci\^encias da Universidade do Porto \\ 
Rua Campo Alegre 687 \\ 4169-007 Porto, Portugal} 
\email{lkadison@fc.up.pt } 
\thanks{}
\date{} 

\begin{abstract}
Let $C_n(A,B)$ be the relative Hochschild bar resolution groups of a subring $B \subseteq A$.
The subring pair has right depth $2n$ if $C_{n+1}(A,B)$ is isomorphic to a direct summand of a multiple of $C_n(A,B)$ as $A$-$B$-bimodules; depth $2n+1$ if the same condition holds only as
$B$-$B$-bimodules.  It is then natural to ask what is defined if this same condition should hold
as $A$-$A$-bimodules, the so-called H-depth $2n-1$ condition.   In particular, the H-depth~$1$
condition coincides with $A$ being an H-separable extension of $B$.  In this paper the H-depth
of semisimple subalgebra pairs is derived from the transpose inclusion matrix, and for QF
extensions it is derived from the odd depth of the endomorphism ring extension. For general extensions  characterizations of H-depth are possible using the 
H-equivalence generalization of Morita theory.  
\end{abstract} 
\maketitle

\section{Introduction}
Given a unital subring $B$ in an associative unital ring $A$ where $1_A = 1_B$,  this paper continues a study of  certain
bimodule conditions on the $n$-fold tensor products $A \otimes_B \cdots \otimes_B A$.  In the papers
\cite{BK, BDK, KS}  the ring extension $A \supseteq B$ is said to have left depth $2$, right depth $2$, or
depth $3$ if the tensor-square has a split bimodule monomorphism into a multiple of $A$ as respectively $B$-$A$-, $A$-$B$- or $B$-$B$-bimodules.  The depth $2$ conditions are interesting
from the point of view of Galois theory, since $\End {}_BA_B$ has a finite projective bialgebroid structure over the centralizer subring $A^B$ and acts naturally on $A$ (e.g., \cite{KS, K2008}).  The depth $3$ condition on a Frobenius extension $A \| B$ is also of Galois-theoretic interest, since
the left regular representation $\lambda: A \into \End A_B := E$ restricts to a ring extension
$B \into E$ having depth $2$ \cite[Theorem 2.5]{LKJPAA}.  In this case the ring $\End {}_BE_A$ is a left coideal subring of 
$\End {}_BE_B$ with good Galois-theoretic properties  of a `` depth-3 tower''
$B \subseteq A \into E$ sketched in \cite[Sections 4,5]{LKJPAA}. 

A similar definition of the ring extension $A \supseteq B$ having left depth $2n$, right depth $2n$
or depth $2n+1$ holds:  there is a split monic of the $(n+1)$-fold tensor product into a multiple
of the $n$-fold tensor product as natural $B$-$A$-, $A$-$B$- or $B$-$B$-bimodules, respectively  \cite{BDK}.  In case
this is a Frobenius extension with surjective Frobenius homomorphism $E: A \rightarrow B$,
the ring extension having depth $n$ embeds in a tower of iterated right endomorphism rings $E_1 \into E_2 \into \cdots$
where $B \into E_{n-3} \into E_{n-2}$ is also a ``depth-3 tower'' \cite{LK}. 

 The minimum
depth $d(B,A)$ realizes each positive integer for complex semisimple algebras $B \subseteq A$
with Bratteli diagram a Dynkin diagram of type $A_n$; see \cite[3.11]{BKK}.  However, for the
group algebras  $B =\C[H]$ and $A = \C[G]$ of a subgroup $H$ of a finite group $G$,
the values of $d(B,A)$ seem to be limited to the odd values and only the even values $\{ 2,4,6 \}$;
see \cite{BDK, BK, BKK, D, F}.  

In the three bimodule-theoretic definitions above of left, right even depth  and odd depth, the fourth case
of $A$-$A$-bimodules has been sidestepped so far, but is taken up in this paper.  We pose 
the question, what is defined on a ring extension $A \supseteq B$ if the $(n+1)$-fold tensor
product has a split $A$-$A$-bimodule monic into a multiple of the $n$-fold tensor product
$A \otimes_B \cdots \otimes_B A$?  This question has classical roots in case $n = 1$; the condition
that $A \otimes_B A$ is a direct summand of $A^n = A \oplus \cdots \oplus A$ (or $nA$ in
additive notation) as natural $A$-bimodules is the condition that $A$ is an H-separable extension of $B$
\cite{H}.  These have an elaborate theory
generalizing Azumaya algebra, where among other things one proves with some commutative algebra
that also ${}_AA_A$ is a direct summand of $A \otimes_B A$, i.e. $A$ is a separable extension of $B$. 
 H-separability is studied in e.g. \cite{H, HS, K, K2001, K2003}.  

In section~2 we define $A \supseteq B$ having H-depth $2n-1$ as the condition just given in terms
of $A$-bimodules on the $(n+1)$-fold and $n$-fold tensor products.  We sketch the general theory
for ring extensions, noting that the minimum H-depth $d_H(B,A)$ and minimum depth $d(B,A)$
differ by at most $2$, if one is finite. We also apply results from \cite{BDK, LK} about when $d(B,A)$
is finite.  In section~3 we restrict to $A$ and $B$ being complex semisimple algebras when information
about the subalgebra structure 
is nicely recorded by weighted bicolored multi-graphs and inclusion matrices that are viewable as homomorphisms
of the $K$-groups $K_0(B) \rightarrow K_0(A)$).  In this case the H-depth is a condition on the transpose
of the inclusion matrix.  In section~4 we note that a Frobenius extension $A \supseteq B$ having
H-depth $2n-1$ occurs precisely when the left regular extension $E \supseteq A$ (seen above) has
depth $2n-1$.  However we note through examples that the depths $d_H(B,A)$, $d(B,A)$
and $d(A,E)$ may differ from one another.  

\section{General theory}

Given a unital associative ring $R$ and unital $R$-modules $M$ and $N$, we write $N \oplus * \cong M^q$ if $N$ is isomorphic to a direct summand in $M^q = M \oplus \cdots \oplus M$ ($q$ times).  Recall that if there is
also a positive integer $s$ such that $M \oplus * \cong N^s$, then $M$ and $N$ are similar, or H-equivalent, as $R$-modules; denoted by $M \h N$, indeed an equivalence relation.  In this case their endomorphism rings $\End M_R$ and $\End N_R$
are Morita equivalent with Morita context bimodules $\Hom (M_R,N_R)$ and
$\Hom (N_R,M_R)$ (with composition as module actions and Morita pairings).   If the category of finitely generated $R$-modules has unique factorization into
indecomposables, then finitely generated $M$ and $N$ have the same indecomposable constituents if and only if
$M$ and $N$ are H-equivalent modules.  If $F$ is an additive endofunctor of the category of $R$-modules, then $M \h N$ implies $F(M) \h F(N)$; which in practice means that H-equivalent bimodules
may replace one another in certain H-equivalences of  tensor products.  

 Throughout this paper, let $A$ be a unital associative ring and $B \subseteq A$ a subring where $1_B =  1_A$.  Note the natural bimodules ${}_BA_B$ obtained by restriction of the natural $A$-$A$-bimodule (briefly $A$-bimodule) $A$, also  to the natural bimodules ${}_BA_A$, ${}_AA_B$ or ${}_BA_B$, which
are referred to with no further notation.  Equivalently we denote the proper ring extension $A \supseteq B$ occasionally by $A \| B$.  (Often results are valid as well for a ring homomorphism $B \rightarrow A$
and its induced bimodules on $A$.)  

Let $C_0(A,B) = B$, and for $n \geq 1$, 
$$ C_n(A,B) = A \o_B \cdots \o_B A \ \ \ \mbox{\rm ($n$ times $A$)} $$
For $n \geq 1$, the $C_n(A,B)$ has a natural $A$-bimodule  structure given by
$a(a_1 \otimes \cdots \otimes a_n)a'  = aa_1 \otimes \cdots \otimes a_na'$.  
Of course, this bimodule structure restricts to $B$-$A$-, $A$-$B$- and $B$-bimodule structures as we may need them.
\begin{definition}
\label{def-depth}
The subring $B \subseteq A$ has H-depth $2n-1 \geq 1$ if for some positive integer $q$,
\begin{equation}
\label{eq: H-depth}
  C_{n+1}(A,B) \oplus * \cong C_n(A,B)^q.
\end{equation}
Note that if a subring has H-depth $2n-1$, it has H-depth $2n+1$ by tensoring the displayed equation
by either $-\otimes_B A_A$ or ${}_AA \otimes_B -$.  If $B \subseteq A$ has a finite H-depth, denote its minimum H-depth by $d_H(B,A)$.  
 \end{definition}

Note that H-depth $1$ occurs when ${}_AA \otimes_B A_A \oplus * \cong {}_AA_A^n$ for some $n \in \Z_+$, which is the
condition that $A$ be an H-separable extension of $B$. (Replacing $A^n$ with the direct sum $A^{(I)}$ for $I$ an arbitrary indexing set in the definition of H-separable extension  does not obtain anything different since
$A \otimes_B A$ is a cyclic $A$-bimodule \cite[Prop. 1.3(4)]{K2008}.) This is a classical generalization of
the Azumaya condition on algebras to ring extensions (see for example \cite{H, HS, K, K2001, K2003}); indeed, one may prove that $A$ is
a separable extension of $B$, so that ${}_AA_A \oplus * \cong {}_AA \otimes_B A_A$ since
the multiplication mapping $\mu: A \otimes_B A \rightarrow A$ splits. Recall too that
a ring extension $A \| B$ is separable if and only if the relative Hochschild cohomology groups $H^n(A,B;M) = 0$ for all $n > 0$ and all
bimodule coefficient modules ${}_AM_A$ \cite{HS}.  

\begin{lemma}
\label{lem-hequiv}
A subring $B \subseteq A$ has H-depth $2n-1$ if and only if $C_n(A,B) \h C_{n+1}(A,B)$ as
$A$-bimodules ($n = 1,2,3,\ldots$). 
\end{lemma}
\begin{proof}
We just noted above that the H-depth  $1$ condition is H-separability, which implies separability, 
and the two conditions together
give $C_2(A,B) \h C_1(A,B)$ as $A$-bimodules.  For $n=2$ the $A$-bimodule epi
$C_3(A,B) \rightarrow C_2(A,B)$ given by $a_1 \otimes a_2 \otimes a_3 \mapsto a_1 a_2 \otimes a_3$
is split by $a_1 \otimes a_2 \mapsto a_1 \otimes 1 \otimes a_2$.  Then $C_2(A,B) \oplus * \cong C_3(A,B)$. An obvious generalization gives that $C_n(A,B) \oplus * \cong C_{n+1}(A,B)$. 
It follows that the condition in the lemma is equivalent to the condition in the definition. 
\end{proof}

Let $C_0(A,B)$ denote the natural $B$-bimodule $B$ itself.  
Recall from \cite{BDK, LK} that a subring $B \subseteq A$ has  right depth $2n$
if \begin{equation}
\label{eq: subringdepth}
C_{n+1}(A,B) \h C_n(A,B)
\end{equation}
as natural $A$-$B$-bimodules; left depth $2n$ if the same condition holds as $B$-$A$-bimodules;
if both left and right conditions hold, it has depth $2n$; 
and depth $2n+1$ if the same condition holds as $B$-bimodules. Note that if the subring
has left or right depth $2n$, it automatically has depth $2n+1$ by restriction to $B$-bimodules. 
Also note that if the subring has depth $2n+1$, it has depth $2n+2$ by
tensoring the H-equivalence by $- \o_B A$ or $A \o_B -$. The \textit{minimum depth} is denoted
by $d(B,A)$; if $B \subseteq A$ has no finite depth, write $d(B,A) = \infty$.   

The dependence of depth and H-depth only on the H-equivalence class of the natural bimodule of a ring extension is made explicit below.  

\begin{lemma}
Suppose $A \supseteq C$ and $B \supseteq C$ are two ring extensions of the same ring.
If the natural $C$-bimodules are H-equivalent, ${}_CA_C \h {}_CB_C$, then $A \supseteq C$
has depth $2n+1$ if and only if $B \supseteq C$ has depth $2n+1$. Suppose moreover
$A \supseteq B \supseteq C$ is a tower of rings.  If $B \supseteq C$ has H-depth $2n-1$ (for $n>1$),
then $A \supseteq C$ has H-depth $2n-1$.
\end{lemma}
\begin{proof}
By the substitution principle for the H-equivalent bimodules ${}_CA_C \h {}_CB_C$, we obtain
$C_n(A,C) \h C_n(B,C)$ as $C$-bimodules for all $n \geq 1$.  Thus, $C_{n+1}(A,C) \h C_n(A,C)$ if and only if $C_{n+1}(B,C) \h C_n(B,C)$ for all $n \geq 0$.  This proves the first statement in the lemma. 

 The second statement
follows from applying  the additive functor$A \otimes_B - \otimes_B A$ (from $B$-bimodules
into $A$-bimodules) to the H-equivalence of $B$-bimodules,  
$C_{n+1}(B,C) \h C_n(B,C)$, cancelling certain trivial tensors to obtain \newline 
$A \otimes_C C_{n-1}(B,C) \otimes_C A \h A \otimes_C C_{n-2}(B,C) \otimes_C A$
as $A$-bimodules.  But $C_n(B,C) \h C_n(A,C)$ as $C$-bimodules for each $n > 0$ by the hypothesis,
so that $C_{n+1}(A,C) \h C_n(A,C)$ as $A$-bimodules.  Thus $A \supseteq C$ has H-depth
$2n-1$ as well.  
\end{proof}
\begin{prop}
If a subring has H-depth $2n-1$, then it has depth $2n$.  If a subring has left or right depth $2n$, then
it has H-depth $2n+1$. As a consequence, 
\begin{equation}
\label{eq: comparison}
|d(B,A) - d_H(B,A) | \leq 2.
\end{equation} 
\end{prop}
\begin{proof}
The first statement follows from restricting the condition on $A$-bimodules for H-depth in the lemma 
to the condition~\eqref{eq: subringdepth} on either $A$-$B$- or $B$-$A$-bimodules.  The second
statement follows for example by tensoring the right depth $2n$ condition from the right
by $ - \otimes_B A_A$ to obtain the H-depth condition. The inequality follows from 
applying the first two statements.
\end{proof}
For example, an H-depth $1$ extension is known to have depth $2$, with associated
Galois structure computed in \cite{K2003}.

\begin{example}
\begin{rm}
Let $B = \C^2 \into A = M_2(\C)$ be given by $(\lambda, \mu) \mapsto \left( \begin{array}{cc}
\lambda & 0 \\
0 & \mu 
\end{array}
\right)
$ for each $\lambda, \mu \in \C$.  Since $B$ is a semisimple subalgebra of the Azumaya algebra $A$,
this is an H-separable extension.  The extension is also normal and depth $2$ by
\cite[Prop.\ 4.3] {BKK}.  By the results of \cite{B} a subalgebra pair of complex semisimple algebras
$B \subseteq A$ has depth $1$ if and only if their centers satisfy $\mathcal{Z}(B) \subseteq \mathcal{Z}(A) $.  In this example this is not the case, whence $d_H(B,A) = 1$ and $d(B,A) = 2$.

In Section~\eqref{sect-semisubpairs} it is noted how to compute depth and H-depth directly from the inclusion matrix $M = \left( \begin{array}{c}
1 \\
1
\end{array}
\right) $ and its transpose. 
\end{rm}
\end{example}
For a $k$-algebra $B$ let $B^e$ denote $B \otimes_k B^{\rm op}$. 
\begin{cor}
\label{cor-finite}
Let $B \subseteq A$ be a subring pair of finite dimensional algebras.  If either $B^e$ has finite representation type or both $A$ and $B$ are group algebras $k[G]$ and $k[J]$ with $J < G$
a subgroup pair, then $d_H(B,A) < \infty$. 
\end{cor}
\begin{proof}
If $B^e$ has finite representation type, it is shown in \cite{LK} that subring depth $d(B,A)$ is bounded
by one plus twice the number of  isomorphism classes of indecomposable $B$-bimodules.  If
$A \supseteq B$ is a group algebra extension, it is shown in \cite{BDK} that $d(B,A)$ is bounded
by twice the index of the normalizer of $J$ in $G$.  
\end{proof}
Recall the relative Hochschild cochain groups with coefficient bimodule ${}_AM_A$ \cite{Hoch}:  they
are denoted and defined by $C^n(A,B;M) = \mbox{\rm Hom}_{B^e}(C_n(A,B), M)$.  
\begin{prop}
Suppose $A \supseteq B$ has H-depth $2n-1 \geq 3$.  
Then $C^{n-1}(A,B;M) \h C^{n-2}(A,B;M)$ as abelian groups.  
\end{prop}
\begin{proof}
Apply the additive functor $\mbox{\rm Hom}_{A^e}( -, M)$ to the H-equivalence in
the lemma above.  Note that $\mbox{\rm Hom}_{A^e}( C_{n+1}(A,B), M) \cong
\mbox{\rm Hom}_{B^e}(C_{n-1}(A,B), M)$. The H-equivalence in the proposition follows.  
\end{proof}
This may be applied to H-depth $3$ and the natural bimodule $M = A$ to obtain the following.
\begin{cor}
For an H-depth $3$ extension $A \supseteq B$, the endomorphism ring $\End {}_BA_B$ is
H-equivalent to the centralizer $A^B$ as modules over the center of $A$.
\end{cor}
We end with a characterization of ring extensions $A \| B$ having H-depth $3$.  
Let $T := (A \otimes_B A)^B$ and note $T \cong \End {}_A A \otimes_B A_A$ via
$t \mapsto (a \otimes a' \mapsto ata')$ with inverse $F \mapsto F(1 \otimes 1)$.  Thus
$T$ is a ring with multiplication $st = t^1 s^1 \otimes s^2 t^2$ and $1_T = 1 \otimes 1$,
and $A \otimes_B A = C_2(A,B)$ is a left $T$-module ($t \cdot (a\otimes a') = ata'$);
in fact a $T$-$A^e$-bimodule.  

In addition $Q := (A \otimes_B A \otimes_B A)^B$ is a right $T$-module via $q \cdot t = t^1 q t^2$.
It is a generator $T$-module via an obvous split epi $Q_T \rightarrow T_T$.

\begin{theorem}
A ring extension $A \| B$ has H-depth $3$ if and only if $Q_T$ is finite generated projective
and the mapping $q \otimes_T (a \otimes a') \mapsto aqa'$ is an isomorphism of $A$-bimodules,
$$
\label{eq: iso}
Q \otimes_T C_2(A,B) \stackrel{\cong}{\longmapsto} C_3(A,B)
$$
\end{theorem}

\begin{proof}
($\Leftarrow$) From $Q_T \oplus * \cong T_T^n$ and tensoring from the right by
$-\otimes_T C_2(A,B) $, one obtains the H-depth $3$ defining condition
\begin{equation}
\label{eq: hd3}
C_3(A,B) \oplus * \cong C_2(A,B)^n
\end{equation}
 as $A$-bimodules.

($\Rightarrow$)  Applying a diagram chase to \eqref{eq: H-depth} where $n=2$, similar to the one for extracting dual bases for a projective module, one finds elements $q_i \in Q$ and $g_i \in \Hom ({}_BA_B, {}_BA \otimes_B A_B)$ ($i = 1,\ldots,q$) such that for all $a_1,a_2, a_3 \in A$,
\begin{equation} \nonumber
a_1 \otimes_B a_2 \otimes_B a_3 = \sum_{i=1}^q a_1 g_i^1(a_2)q_i g_i^2(a_2) a_3
\end{equation}
where $g_i(a) = g_i^1(a) \otimes_B g_i^2(a)$ is a type of Sweedler notation suppressing a possible summation.

It follows that the mapping defined in the theorem has inverse mapping defined on
$c \in C_3(A,B)$ by
$$ c \longmapsto \sum_{i=1}^q q_i \otimes_T c^1 g_i(c^2) c^3 $$
where we again use Sweedler-type notation for $c$.  

(Alternatively the mapping in the theorem is an isomorphism by applying \cite[Theorem 2.20]{K}; it is in particular the left vertical isomorphism in Figure~1.)

Finally $Q \cong \Hom ({}_AC_2(A,B)_A, {}_AC_3(A,B)_A)$ via $q \mapsto (a\otimes a' \mapsto aqa')$
with inverse $G \mapsto G(1 \otimes 1)$.  But this Hom-group is a Morita context bimodule
for the Morita equivalent rings $\End {}_AC_2(A,B)_A \cong T$ and $\End {}_AC_3(A,B)_A$
stemming from the H-equivalence of $C_2(A,B)$ and $C_3(A,B)$.  Whence
$Q_T$ is finite projective.  
\end{proof}

Similar characterizations of (H-) separable extensions are obtained in \cite[Theorems 4.1, 4.2]{KK}.

Given a ring extension $A \| B$, let $E$ denote its right endomorphism ring $\End A_B$.  Its natural
$A$-bimodule structure is given by $(afa')(x) = af(a'x)$.  Recall that a right $B$-module $V$ coinduces
to an $A$-module $\Hom (A_B, V_B)$; e.g., $V = \Res^A_B A$ has coinduced module $E_A$.  

\begin{prop}
\label{prop-coinduction}
Suppose $A \|B$ has finitely generated projective module $A_B$.
Then $d_H(B,A) \leq 3$ if and only if $\Coind^A_B \Res^A_B E \oplus * \cong E^n$
as $A$-bimodules for some $n \in \Z_+$.
\end{prop}
\begin{proof}
($\Rightarrow$)  This direction in the proof does not require finite projectivity of $A_B$.  
Apply the additive endofunctor $\Hom (-_A, A_A)$ to the isomorphism~\eqref{eq: hd3}
in the category of $A$-bimodules.  Note that $\Hom (A \otimes_B A_A, A_A) \cong E$
via $F \mapsto F(- \o_B 1_A)$, and that $\Hom (C_3(A,B)_A, A_A) \cong \Hom (A \o_B A_B, A_B)$
via $F \mapsto F(- \o_B - \o_B 1_A)$.  This obtains
\begin{equation}
\label{eq: hd3'}
 \Hom (A \otimes_B A_B, A_B) \oplus * \cong E^n
\end{equation}
as natural $A$-bimodules. 

By the adjoint relation between homming and tensoring, $\Hom (A \o_B A_B, A_B) \cong \Hom (A_B, E_B)$ via $F \mapsto (a \mapsto F(a \otimes -))$.  The $A$-bimodule isomorphism in the theorem
follows.

($\Leftarrow$) If $A_B$ is finite projective, then $A \otimes_B A_A$ is also finite projective,
therefore reflexive. Then applying $\Hom ({}_A-, {}_AA)$ to the isomorphism~\eqref{eq: hd3'}
obtains the defining eq.~\eqref{eq: hd3} for H-depth $3$ extension $A \| B$.  
\end{proof}


\section{Subalgebra pairs of complex semisimple algebras}
\label{sect-semisubpairs}

Given a matrix $M$ we let $M^t$ denote its transpose matrix in this section.  Two~$r$-by-$s$ matrices
$M,N$ satisfy an inequality $M \leq N$ if each pair of $(i,j)$-entries satisfy  $M_{ij} \leq N_{ij}$.  The $r \times s$ zero matrix is denoted by $0$. 

Let $B \subseteq A$ be a subring pair of semisimple complex algebras.
Then the minimum depth $d(B,A)$ may be computed from the inclusion matrix $M$, equivalently an $r$-by-$s$ induction-restriction table of $r$ $B$-simples induced to non-negative integer linear combination of $s$ $A$-simples along rows;
by Frobenius reciprocity, columns show restriction of $A$-simples in terms of $B$-simples. The procedure to obtain $d(B,A)$ given in the paper \cite{BKK} is the following:  let $M^{[2n]} = (MM^t)^n$ and $M^{[2n+1]} = M^{[2n]} M$ (and $M^{[0]} = I_r$), then the matrix $M$ has depth $n \geq 1$
if  for some $q \in \Z_+$ 
\begin{equation}
\label{eq: dn inequality}
M^{[n+1]} \leq qM^{[n-1]}
\end{equation}

Note that if $M$ has depth $n$, it has depth $n+1$ by multiplying the inequality by $M \geq 0$.  
  The minimum depth of $M$ is equal to $d(B,A)$ (see \cite{BDK}). One may note  that
$d(B,A) \leq 2d-1$ where $MM^t$ has  minimal polynomial of degree $d$ \cite{BKK}.    Thus $d(B,A) < \infty$ from which it follows from inequality~\eqref{eq: comparison} that $d_H(B,A) < \infty$ for a complex semisimple subalgebra pair $B \subseteq A$; alternatively, note that $B^e$ has finite representation type and apply Corollary~\eqref{cor-finite}

Let $Z(M)$ be the number of zero entries in an $r \times s$ matrix $M$ of nonnegative integers:  since an inclusion matrix $M$ has nonzero rows and columns, $rs - \max \{r,s \} \geq Z(M) \geq 0$.  Since $M \geq 0$,
it follows that $Z(M^{[n]})$ are decreasing nonnegative integers with increasing odd or even bracketed powers of $M$:
 \begin{equation}
\nonumber
Z(M^{[n+2]}) \leq  Z(M^{[n]}).
\end{equation}
It is quite easy to see that $d(B,A)$ is the least $n$ for which $Z(M^{[ n+2]}) = Z(M^{[n]})$
where $M$ is the inclusion matrix for $B \subseteq A$ (cf.\ \cite{BKK}).  
 
 In terms of the bipartite graph of the inclusion $B \subseteq A$, $d(B,A)$ is the lesser of the minimum
odd depth and the minimum even depth \cite{BKK}.  The matrix $M$ is an incidence matrix
of this bipartite graph if all entries greater than $1$ are changed to $1$, while zero entries are retained as $0$: let the $B$-simples
be represented by $r$ black dots in a bottow row of the graph, and $A$-simples by $s$ white dots in a top row, connected by edges joining black and white dots (or not) according to the $0$-$1$-matrix entries obtained from $M$.  The minimum odd depth of the bipartite graph is $1$ plus the diameter in edges of
the row of black dots (indeed an odd number) \cite[3.6]{BKK}, while the minimum even depth is $2$ plus the largest of the diameters of the bottom row where a subset of black dots under one white dot is identified with one another \cite[3.10]{BKK}.    

\begin{example}
\label{ex-s-4}
\begin{rm}
Let $A = \C [S_4]$, the complex group algebra of the
permutation group on four letters, and $B = \C [S_3]$.  The inclusion diagram pictured below with the degrees of the irreducible representations, is determined
from the character tables of $S_3$ and $S_4$ or the branching rule (for the Young diagrams labelled by the partitions of $n$ and representing the irreducibles of $S_n$).  
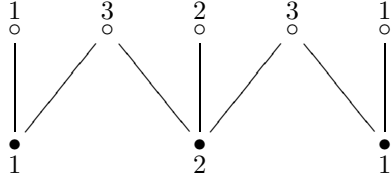
\[
\begin{xy}
\xymatrix{
\overset{\displaystyle 1}{\circ} \ar@{-}[d]& \overset{\displaystyle 3}{\circ} \ar@{-}[ld] \ar@{-}[rd] & \overset{\displaystyle 2}{\circ} \ar@{-}[d]&
\overset{\displaystyle 3}{\circ} \ar@{-}[ld] \ar@{-}[rd]&\overset{\displaystyle 1}{\circ} \ar@{-}[d]\\
 \mathop{\bullet}\limits_{\displaystyle 1} & &\mathop{\bullet}\limits_{\displaystyle 2} & & \mathop{\bullet}\limits_{\displaystyle 1}}
\end{xy}
\]
This graph has minimum odd depth 5 and minimum even depth 6, whence $d(B,A) = 5$. 
Alternatively, the inclusion matrix $$M = \left( \begin{array}{ccccc}
1 & 1 & 0  &  0 & 0  \\
0   &  1  & 1  &  1 &  0 \\
 0  &  0  &  0  &  1 &  1
\end{array}
\right) $$
satisfies the inequality~\eqref{eq: dn inequality} only when $n \geq 5$.  
\end{rm}
\end{example}

Recall orthonormal expansion of a vector in an inner product space with orthonormal bases
$e_1,\ldots, e_n$:  $v = \sum_{i=1}^n \bra v, e_i \ket e_i$.  
\begin{theorem}
Suppose $B \subseteq A$ is a subalgebra pair of complex semisimple algebras with  inclusion matrix $M$. Then $B \subseteq A$ has H-depth $2n-1$ iff the symmetric matrix $S = M^tM$ satisfies
\begin{equation}
\label{eq: H-depthmatrix}
S^n \leq q S^{n-1}
\end{equation}
 for some $q \in \Z_+$.  
\end{theorem}
\begin{proof}
($\Rightarrow$) Suppose the isomorphism classes of simple right $A$-modules are represented by $\mbox{\rm Irr}(A) = \{ V_1,\ldots,V_s \}$,
and $\mbox{\rm Irr}(B) = \{ W_1,\ldots, W_r \}$.  Then $M$ is an $r \times s$-matrix with entries
$m_{ij}$ such that restriction of modules, $V_j \downarrow_B \cong \oplus_{i=1}^r m_{ij}W_i$,
and induction of modules, $W_i \uparrow^A \cong \oplus_{j=1}^s m_{ij} V_j$ in additive notation.
With the usual inner product of $A$- or $B$-modules, where $\bra W_i, W_j \ket = \delta_{ij} = \bra V_i, V_j \ket$, we note that $m_{ij} = \bra V_j \downarrow_B \, , \, W_i \ket_B$ and 
\begin{eqnarray}
\label{eq: tys}
 (M^tM)_{ij} = & \sum_k m_{ki}m_{kj} = &  \sum_{k,t} m_{ki} m_{tj} \bra W_k, W_t \ket \\
& =  \bra V_i \downarrow_B, V_j \downarrow_B \ket_B = &  \bra V_i \downarrow_B \uparrow^A, V_j \ket_A  \nonumber
\end{eqnarray}
the last step being Frobenius reciprocity.

Now if ${}_AA \otimes_B A_A \oplus * \cong {{}_AA_A}^q$, then $V_i \otimes_B A_A \oplus * \cong
q V_i$ after standard tensor cancellations. Thus,  $\Ind^A_B \Res^A_B V_i$ is isomorphic
to a direct summand of a multiple of $V_i$.  Denoting the endo-operator on $A$-modules by
$T = \Ind^A_B \Res^A_B$, we similarly see that the H-depth $2n-1$ condition in Definition
 \eqref{eq: H-depth} implies that $T^n(V_i) \oplus * \cong qT^{n-1}(V_i)$.  Thus for each $V_j$ there is the
inequality 
\begin{equation}
\label{eq: ineq}
\bra T^n(V_i) \, , \, V_j \ket \leq  q \bra T^{n-1}(V_i) \, , \, V_j \ket.
\end{equation}   

But the matrix of $T$ in terms of the basis $\{ V_1,\ldots, V_s \}$ is given by $S= M^tM$ according
to the computation in Eq.~\eqref{eq: tys}.  Thus, the inequality~\eqref{eq: ineq} becomes $(S^n)_{ij} \leq q (S^{n-1})_{ij}$ for all $i,j = 1,\ldots,s$.  

$(\Leftarrow )$ If the inclusion matrix $M$ of semisimple subalgebra pair $B \subseteq A$ satisfies $(M^tM)^n \leq q(M^tM)^{n-1}$ for
some $q \in \Z_+$, then $\bra T^n(V_i) \, , \, V_j \ket \leq q\bra T^{n-1}(V_i) \, , \, V_j \ket $
for all $A$-simples $V_j$. Via unique module decomposition into simples, we find a monic natural
transformation $T^n \hookrightarrow q T^{n-1}$ of endofunctors on the category $A$-Mod. 
Now  $A$ and therefore $A^e$ are separable $\C$-algebras, so as in [\cite{KK}, Theorem 2.1(6), pp. 3107-3108], we apply the natural monic to the right regular module $A_A$, apply the natural transformation property to all left multiplications $\lambda_a$ ($a \in A$), and note that $C_{n+1}(A,B) \hookrightarrow qC_n(A,B)$ splits by Maschke as an $A$-bimodule
monic. Hence $A$ has H-depth~$2n-1$ over its subalgebra $B$. 
\end{proof}

Following \cite{BKK} we say that  an $r \times s$-matrix $M$ of nonnegative integers, and nonzero
rows and columns, has depth $n$ if the inequality~\eqref{eq: dn inequality} is satisfied. Similarly define the minimum depth $d(M)$ to be
the least positive integer $n$ for which $M$ has depth $n$.  The matrix $M$ always has a finite depth,
bounded by degree of the minimum polynomial of the symmetric matrix $MM^t$.   

\begin{cor}
With the hypotheses of the theorem, the minimum H-depth of $B \subseteq A$ and the minimum
depth of the transpose inclusion matrix satisfy
\begin{equation}
\label{eq: mt}
0 \leq d_H(B,A) - d(M^t) \leq 1.
\end{equation}
\end{cor}
\begin{proof}
The matrix $M^t$ is the inclusion matrix for the endomorphism ring extension $A \into E$ (given by $a \mapsto \lambda_a$
where $\lambda_a(x) = ax$ for all $a,x \in A$), also a subalgebra pair of complex
semisimple algebra inclusions \cite[3.13 and above]{BKK}.  Note that the inequality~\eqref{eq: H-depthmatrix} is the depth~$2n-1$ condition on the transpose
matrix $M^t$.  Thus if $d(M^t)$ is odd, then $d_H(B,A) = d(M^t)$, and if $d(M^t)$ is even, then
$d_H(B,A) = d(M^t) +1$.  
\end{proof} 

It follows from \cite[Prop. 2.5]{BKK} that for any inclusion matrix $M$,
\begin{equation}
\label{eq: ms}
|d(M^t) - d(M)| \leq 1
\end{equation}

\begin{example}
\begin{rm}
Continuing the example $B = \C [S_3] \subseteq \C [S_4] = A$ above, where we obtained
$M$ and its depth $d(M) = d(B,A) = 5$, we may continue computing bracketed powers of
$M^t$ to obtain $d(M^t) = 6$ and $d_H(B,A) = 7$. (Alternatively, the depth $d(M^t) = d(A,E)$
may be computed as the depth of the graph in Example~\eqref{ex-s-4} reflected about the top row of white dots (so the 
$3$ black dots land above the $5$ white dots with weights $4$, $8$ and $4$.  See \cite[3.12]{BKK} for
the explanation in terms of Morita equivalence.)  This shows that the various inequalities
in~\eqref{eq: comparison},~\eqref{eq: mt},~\eqref{eq: ms} and~\eqref{eq: endo} may not be improved.  
\end{rm}
\end{example}

\begin{example}
\begin{rm}
Let $H_8$ denote the eight-dimensional self-dual semisimple Hopf algebra of Masuoka \cite{M} and
Kac-Paljutkin \cite{KP} over
an algebraically closed field $k$ of characteristic zero with generators $x,y,z$ and relations $x^2 = y^2 = 1$, $xy = yx$, $zx = yz$, $zy = xz$
and $2z^2 = 1 + x + y - xy$.  Its coalgebra structure is determined by $\cop(x) = x \otimes x$,
$\eps(x) = 1$, $S(x) = x$, $\cop(y) = y \otimes y$, $\eps(y) = 1$, $S(y) = y$, and
$\cop(z) = \frac{1}{2}((1 + y) \otimes 1 + (1-y)\otimes x)(z \otimes z)$, $\eps(z) = 1$ and $S(z) = z$.

Burciu \cite{B2006} computes the irreducible characters of $H_8$ and its Drinfeld double
$D(H_8)$ as well as the induced representations.  The space of irreducible characters of $H_8$
is $5$-dimensional with four linear and one degree $2$ irreducible characters: $H_8 \cong k^4 \times M_2(k)$.  The space of irreducible characters of $D(H_8)$ is $22$-dimensional with
$D(H_8) \cong k^8 \times M_2(k)^{14}$ \cite[pp.\ 501-503]{B2006}.  Thus, $M$ is a $5 \times 14$
matrix determined by the table \cite[p.\ 503]{B2006}; its bracketed square is computed to be
\begin{equation} \nonumber
MM^t = \left( \begin{array}{ccccc}
5 & 1 & 1 & 1 & 0 \\
1 & 5 & 1 & 1 & 0 \\
1 & 1 & 5 & 1 & 0 \\
1 & 1 & 1 & 5 & 0 \\
0 & 0 & 0 & 0 & 8 
\end{array}
\right)
\end{equation}
From this matrix and the table it follows that $d(H_8,D(H_8)) = 3$.  The matrix $M^tM$ is an order $22$ square matrix $\mathcal{S}$ that does not
satisfy $Z({S}^2) = Z(\mathcal{S})$, whence $d_H(H_8,D(H_8)) = 5$.   
\end{rm}
\end{example}

\section{Frobenius extensions}
\label{sec: FE}

A Frobenius extension $A \supseteq B$ is characterized by any of the following four conditions
\cite{K}.
First, that $A_B$ is finite projective and ${}_BA_A \cong \Hom (A_B, B_B)$. Secondly,
that ${}_BA$ is finite projective and ${}_AA_B \cong \Hom ({}_BA, {}_BB)$.  Thirdly,
that coinduction and induction of right (or left) $B$-modules is naturally equivalent.  Fourth,
there is a Frobenius coordinate system $(E: A \to B; x_1,\ldots,x_m,$ $y_1,\ldots,y_m \in A)$, which satisfies
\begin{equation}
\nonumber
E \in \Hom ({}_BA_B, {}_BB_B), \ \ \sum_{i=1}^m E(ax_i)y_i = a = \sum_{i=1}^m x_i E(y_i a) \ \ (\forall a \in A).
\end{equation}
The so-called dual bases equations may be used to show that $\sum_i x_i \o y_i \in (A \o_B A)^A$
\cite{K}. 
A Frobenius extension $A \supseteq B$ has generator module $A_B$ iff it has a surjective
Frobenius homomorphism $E: A \rightarrow B$ \cite{LK}.  Well-known examples of Frobenius
extensions are group algebra extensions, unimodular Hopf algebra extensions, and  projective subalgebra pair
of symmetric algebras \cite{K}.

More generally a ring extension $A \supseteq B$ is a  \textit{QF extension} if both $A_B$ and ${}_BA$ are finite projective,
and the natural bimodules are H-equivalent: ${}_AA_B \h {}_A\Hom ({}_BA,{}_BB)_B$
and ${}_BA_A \h {}_B\Hom (A_B,B_B)_A$ \cite{BM1}.  A Frobenius extension $ A \supseteq B$
is a QF extension since it is left and right finite projective and satisfies the stronger
conditions that $A$ is \textit{isomorphic} to its right $B$-dual $A^*$ and its left $B$-dual ${}^*A$ as natural 
$B$-$A$-bimodules, respectively $A$-$B$-bimodules; the more precise definition are given in the next section.  QF extensions that are not Frobenius extensions may be found among weak Hopf algebras
over their separable base algebras, and matrix examples by Morita \cite{M67}. 

A Frobenius (or QF) extension $A \supseteq B$ enjoys an \textit{endomorphism ring theorem} \cite{BM1, M67}, which shows that $ E := \End A_B \supseteq A$ is a Frobenius (respectively, QF) extension,  where the default ring homomorphism $A \to E$ is understood to be
the left multiplication mapping $\lambda: a \mapsto \lambda_a$ where  $\lambda_a(x) = ax$.   
It is worth noting that $\lambda$ is a left split $A$-monomorphism (by evaluation at $1_A$) so ${}_A E$ is a generator.  

The endomorphism ring $E$ is isomorphic to $A \otimes_B A$ as natural $A$-bimodules
via $f \mapsto \sum_{i=1}^m f(x_i) \otimes_B y_i$ with inverse $a \o_B a' \mapsto \lambda_a \circ
E \circ \lambda_{a'}$ where $\lambda_x$ denotes left multiplication on $A$ by an element $x \in A$.
More generally, coinduction is naturally isomorphic to induction of $B$-modules to $A$-modules;
so that given a right $B$-module $V$, there are similarly defined isomorphisms $V \o_B A \cong
\Hom (A_B, V_B)$.  For a QF extension $A \| B$, the isomorphism is replaced with an H-equivalence:
$V \o_B A \h \Hom (A_B,V_B)$ \cite{LK}.
An immediate consequence of this is a specialization of Proposition~\eqref{prop-coinduction}
to QF extensions with a different proof.
\begin{prop}
Suppose $A \| B$ is a QF extension,  Then $A \| B$ has H-depth $3$ if and only if  $E \otimes_B A \h E$ as $A$-bimodules (alternatively, $\Ind^A_B \Res^A_B E \h E$ as $A$-bimodules).
\end{prop}
\begin{proof}
This follows from $E \h A \otimes_B A$ as $A$-bimodules and Lemma~\eqref{lem-hequiv} in the
case $n = 2$.  
\end{proof}
\begin{theorem}
Suppose $A \| B$ is a Frobenius extension or a QF extension.  Then $A \| B$ has H-depth $2n-1$ if and only if the
endomorphism ring extension $E \| A$ has depth $2n-1$. 
\end{theorem}
\begin{proof} 
If $A \| B$ is a Frobenius extension, we have $E \cong A \otimes_B A$ as natural $A$-bimodules. Then \begin{equation}
\label{eq: sweet}
C_n(E,A)  \cong (A\otimes_B A) \otimes_A \cdots \otimes_A (A \otimes_B A)
\end{equation}
($n$ times $E$ and $n-1$ times $A$), so that $C_n(E,A) \cong C_{n+1}(A,B)$
since there remain $2n - (n-1)$ $A$'s after a standard tensor cancellation.  
Thus as $A$-bimodules $C_{n+1}(A,B) \h C_n(A,B)$ if and only if $C_n(E,A) \h C_{n-1}(E,A)$
as $A$-modules, the latter being condition \eqref{eq: subringdepth} for $E \| A$ having depth $2n-1$. 

If $A \| B$ is QF extension, then $A \otimes_B A \h E$ as $A$-bimodules \cite{LK}, and the same proof
carries through with isomorphisms replaced by H-equivalences. 
\end{proof} 

Since H-depth only assumes odd integer values, it follows from the theorem that
\begin{equation}
\label{eq: endo}
| d_H(B,A) - d(A,E)| \leq 1.
\end{equation}

Recall that a ring extension $A \| B$ is centrally projective if ${}_BA_B \oplus * \cong {}_BB_B^n$
for some $n \in \Z_+$,
and a split extension if ${}_BB_B \oplus * \cong {}_BA_B$.  The ring extension $A \| B$ has depth $1$ iff it is a centrally projective, split extension (cf. \cite{LK, BK2}.
  \begin{cor}
A Frobenius extension is H-separable if and only if its right endomorphism ring extension has depth $1$.
\end{cor}
It was noted in \cite{LK2004} that a group algebra extension $A = \C G \supseteq \C J = B$ that is H-separable
is necesarily trivial: $G = J$.  In \cite{KK} it was proven that group algebra extension $\C G \supseteq \C J$ 
has depth $2$ if and only if $J$ is a normal subgroup of $G$ (and the same result holds for any
base ring by \cite{BDK}).  Again let $E = \End A_B$.  
\begin{cor}
Suppose $J$ is a proper normal subgroup of a finite group $G$.  Then $d(A,E) = 2$ 
\end{cor}

With the Sweedler $A$-coring $A \otimes_B A$  of a ring extension $A \| B$ in mind, eq.~\eqref{eq: sweet} suggests a definition
of H-depth of an $A$-coring $\mathcal{C}$ that generalizes H-depth of a ring extension.
An $A$-coring $\mathcal{C}$ has H-depth $2n-1$ if $\mathcal{C}^{\otimes_A \, n} \h \mathcal{C}^{\otimes_A \, n-1}$ as $A$-bimodules for $n \geq 1$ and $\mathcal{C}^0 = A$. 

 Depth
of an $A$-coring $\mathcal{C}$ with grouplike $g \in \mathcal{C}$ is similarly defined.  
Let $B = \mathcal{C}^g$ be the invariant subring of $A$.  Then $\mathcal{C}$ has
depth $2n+1$, left depth $2n$ or right depth $2n$ if $\mathcal{C}^{\otimes_A \, n} \h
\mathcal{C}^{\otimes_A \, n-1}$ as respectively $B$-$B$-, $B$-$A$- or $A$-$B$-bimodules,
where $\mathcal{C}^{-1} = B$.
The theory of corings, grouplikes, Sweedler corings and ring extensions having depth $2$ are to be found in the book \cite{BW}.  
Depth and H-depth of corings will be investigated in a future paper.

\subsection{Acknowledgements}  The author thanks  P. Carvalho, C. Lomp, R.~Wisbauer
and C.~Young for discussions related to this paper.  Research in this paper was funded by the European Regional Development Fund through the programme {\small COMPETE} 
and by the Portuguese Government through the FCT  under the project \newline
 PE-C/MAT/UI0144/2011.

\end{document}